\documentclass[10pt,twoside]{siamart1116}

\usepackage[english]{babel}
\usepackage{graphicx,epstopdf,epsfig}
\usepackage{amsfonts,epsfig,fancyhdr,graphics, hyperref,amsmath,amssymb, xfrac, arydshln}
\usepackage[T1]{fontenc}

\newcommand{\HE}{Name of Handling Editor}
\newcommand{\DoS}{Month/Day/Year}
\newcommand{\DoA}{Month/Day/Year}
\newcommand{\CA}{Name of Corresponding Author}

\newcommand{\Names}{E.\ Y.\ S.\ Chan, R.\ M.\ Corless, L.\ Rafiee Sevyeri}
\newcommand{\Title}{Generalized Standard Triples\\
for Algebraic Linearizations of Matrix Polynomials}

\newcommand{\diag}[1]{\ensuremath{\mathrm{diag}{#1}}}
\newcommand{\cof}[1]{\ensuremath{\boldsymbol{P}_{#1}}}
\newcommand{\scof}[1]{\ensuremath{p_{#1}}}
\newcommand{\fieldF}{\mathbb{F}}

\renewcommand{\proofbox}{$\natural$}

\newtheorem{remark}[theorem]{Remark}
\newtheorem{Proposition}[theorem]{Proposition}
\newtheorem{example}[theorem]{Example}
\newtheorem{Definition}[theorem]{Definition}

\usepackage{bm}

\newcommand{\mat}[1]{\ensuremath{\boldsymbol{#1}}}

\newcommand{\A}{\ensuremath{\mat{A}}}
\newcommand{\B}{\ensuremath{\mat{B}}}
\newcommand{\C}{\ensuremath{\mat{C}}}
\newcommand{\D}{\ensuremath{\mat{D}}}
\newcommand{\E}{\ensuremath{\mat{E}}}
\newcommand{\F}{\ensuremath{\mat{F}}}
\newcommand{\G}{\ensuremath{\mat{G}}}

\renewcommand{\H}{\ensuremath{\mat{H}}}
\newcommand{\I}{\ensuremath{\mat{I}}}
\newcommand{\J}{\ensuremath{\mat{J}}}
\renewcommand{\L}{\ensuremath{\mat{L}}}

\renewcommand{\P}{\ensuremath{\mat{P}}}
\newcommand{\Q}{\ensuremath{\mat{Q}}}
\newcommand{\R}{\ensuremath{\mat{R}}}

\newcommand{\T}{\ensuremath{\mat{T}}}
\newcommand{\U}{\ensuremath{\mat{U}}}
\newcommand{\V}{\ensuremath{\mat{V}}}
\newcommand{\X}{\ensuremath{\mat{X}}}
\newcommand{\Y}{\ensuremath{\mat{Y}}}

\newcommand{\Ph}{\ensuremath{\boldsymbol{\Phi}}}
 
\renewcommand{\theequation}{\thesection.\arabic{equation}}
\renewcommand{\emph}[1]{\textsl{#1}}
\renewtheorem{theorem}{Theorem}[section]

\begin{document}


\setcounter{page}{1}

\thispagestyle{empty}

 \title{\Title\thanks{Received
 by the editors on \DoS.
 Accepted for publication on \DoA. 
 Handling Editor: \HE. Corresponding Author: \CA}}

\author{
Eunice Y.\ S.\ Chan\thanks{Centre for Medical Evidence, Decision Integrity and Clinical Impact (MEDICI Centre), Department of Anesthesia and Perioperative Medicine, Schulich School of Medicine and Dentistry,
Western University (echan295@uwo.ca)}
\and
Robert M.~Corless\thanks{Ontario Research Centre for Computer Algebra,
School of Mathematical and Statistical Sciences, 
Western University (rcorless@uwo.ca)}
\and 
Leili Rafiee Sevyeri\thanks{Ontario Research Centre for Computer Algebra,
School of Mathematical and Statistical Sciences, 
Department of Applied Mathematics, 
Western University (lrafiees@uwo.ca)}
}

\markboth{\Names}{\Title}

\maketitle

\begin{abstract}
We define \emph{generalized standard triples} $\X$, $\Y$, and $L(z) = z\C_{1} - \C_{0}$, where $L(z)$ is a linearization of a regular matrix polynomial $\P(z) \in \mathbb{C}^{n \times n}[z]$, in order to use the representation $\X(z \C_{1}~-~\C_{0})^{-1}\Y~=~\P^{-1}(z)$ which holds except when $z$ is an eigenvalue of $\P$. This representation can be used in constructing so-called \emph{algebraic linearizations} for matrix polynomials of the form $\H(z) = z \A(z)\B(z) + \C \in \mathbb{C}^{n \times n}[z]$ from generalized standard triples of $\A(z)$ and $\B(z)$. 
This can be done even if $\A(z)$ and $\B(z)$ are expressed in differing polynomial bases. Our main theorem is that $\X$ can be expressed using the coefficients of the expression $1 = \sum_{k=0}^\ell e_k \phi_k(z)$ in terms of the relevant polynomial basis.
For convenience, we tabulate generalized standard triples for orthogonal polynomial bases, the monomial basis, and Newton interpolational bases; for the Bernstein basis; for Lagrange interpolational bases; and for Hermite interpolational bases. We account for the possibility of common similarity transformations. 
\end{abstract}


\section{Introduction} \label{intro-sec}

A \textsl{matrix polynomial}  $\P(z) \in \fieldF^{m\times n}[z]$ is a polynomial in the variable $z$ with coefficients that are $m$ by $n$ matrices with entries from the field $\fieldF$. We will use $\fieldF = \mathbb{C}$, the field of complex numbers, in this paper. Typically an expression in the monomial basis $\phi_k(z) = z^k$ is given for $\P(z)$, and often only \textsl{regular} matrix polynomials are considered, that is, with $m=n$ (we will use $n$ for the dimension) and where $\det\P(z)$ is not identically zero. 
Matrix polynomials have many applications and their study is of both classic and ongoing interest.  See the classic work~\cite{gohberg2009matrix} and the surveys~\cite{guttel2017nonlinear} and~\cite{mackey2015polynomial} for theory and applications.

In this paper, as is done in~\cite{amiraslani2008linearization}, we consider the case when \textsl{any} polynomial basis $\phi_k(z)$ is used. We require that the set $\{ \phi_k(z) \}$ for $ 0 \le k \le \ell$ forms a basis for polynomials of grade~$\ell$. The word ``grade'' is short for ``degree at most.'' 
Thus, we write our regular matrix polynomial as
\begin{equation}
    \P(z) = \sum_{k=0}^\ell \P_k \phi_k(z)\>,
\end{equation}
where the matrices $\P_k \in \fieldF^{n\times n}$ are square, and the \textsl{degree}\footnote{The \emph{degree} of a matrix polynomial $\P(z)$ is defined as follows.  If $\P(z)$ is identically zero, the degree is $-\infty$.  Otherwise, if $\P(z) = \sum_{k=0}^\ell \A_k z^k$ expressed in the monomial basis and $\A_{\ell}$ is not the zero matrix, then the degree of $\P(z)$ is $\ell$.} of $\P(z)$ is at most~$\ell$ ($\P(z)$ has grade $\ell$).   

The notion of ``grade'' is useful even for the monomial basis, but it is especially useful if the basis is an interpolational basis or the Bernstein basis $\phi_k(z) = B^n_k(z) = \binom{n}{k}z^k(1-z)^{n-k}$, when the degree of the polynomial may not be clear from the data.  

For more information about matrix polynomials, consult~\cite{Liesen:HLA:2013}. See also~\cite{mackey2015polynomial}, \cite{guttel2017nonlinear}, and consult the seminal book~\cite{gohberg2009matrix}. 
Linearizations using different polynomial bases were first systematically studied in \cite{amiraslani2008linearization}; for a more up-to-date treatment see~\cite{BuenoCachadina2020}. 
Some other recent papers of interest include~\cite{betcke2013nlevp}, \cite{mackey2006vector}, \cite{dopico2018block}, \cite{Fassbender2017}, and~\cite{robol2017framework}; this is a very active area. 
See also~\cite{al2012standard}.  In that paper, standard triples for structured matrices are studied. In~\cite{corless2021equivalence} several proofs are given of strict equivalence of various linearizations to the standard second linearization for the monomial basis, which we will use without much further comment in this paper.

\subsection{Organization of the Paper}
In Section~\ref{notation-subsec}, we establish notation, give the definitions of algebraic linearization and of generalized standard triples. We define this last in Definition~\ref{def:GST} with reference to the representation in Equation~\eqref{eq:resolventrepresentation}.
In Section~\ref{sec:AlgebraicLinearization}, we show how to use a generalized standard triple in the construction of algebraic linearizations.  We also give a proof, by construction of the necessary equivalence matrices $\E(z)$ and $\F(z)$, that algebraic linearizations are \emph{local} linearizations (a very useful notion from~\cite{Dopico2020}) in the same sense that its components are.  In particular, if the components $\A$ and $\B$ are linearizations in the sets $\Sigma_{\A}$ and $\Sigma_{\B}$ respectively, then the ``algebraic linearization'' is a linearization in the intersection $\Sigma_{\A} \cap \Sigma_{\B}$.

In Section~\ref{sec:universal}, we prove our main result, giving a universal expression for generalized standard triples for the linearizations using the polynomial bases that appear in Sections~\ref{recur-subsec},~\ref{bernstein-subsec}, and~\ref{lagrange-subsec}. 
In Section~\ref{triples-sec}, we introduce all the polynomial bases by scalar examples.
We sum up in the final section.

\subsection{Notation and Definitions} \label{notation-subsec}

Two matrix polynomials $\P_1(z)$ and $\P_2(z)$ are called \textsl{unimodularly equivalent} if there exist unimodular matrix polynomials (that is, matrix polynomials with constant nonzero determinant) $\E(z)$ and $\F(z)$ with $\P_1(z) = \E(z)\P_2(z)\F(z)$.
A matrix pencil $\L(z) := z \C_1 - \C_0$ is called a \textsl{linearization} of the matrix polynomial $\P(z)$ if both $\C_1$ and $\C_0$ are of dimension $N\ge n$ and $\L(z)$ is unimodularly equivalent to the block diagonal matrix $\diag(\P(z),\I_{N-n})$. 
Two linearizations $\L_{m}(z)$ and $\L_\phi(z)$ are called \textsl{strictly} equivalent if the corresponding matrices are equivalent in the following stronger sense: $\C_{1,m} = \E \C_{1,\phi} \F$ and $\C_{0,m} = \E(z)\C_{0,\phi}\F(z)$, with the same invertible (constant) matrices $\E$ and $\F$.

In the paper~\cite{Dopico2020} we find the powerful notion of a \emph{local} linearization on a subset $\Sigma \subset \fieldF$.  In the notation of their Definition 2.1, two \emph{rational} matrices $\G_1(z)$ and $\G_2(z)$ are said to be equivalent in a nonempty set $\Sigma$ if there exist matrices $\R_1(z)$ and $\R_2(z)$ each nonsingular in $\Sigma$ with $\R_1(z)\G_1(z)\R_2(z) = \G_2(z)$.  This notion, intended for use with matrix rational functions, is also useful for matrix polynomials, as we will see. 

There is a related idea to linearization, that of ``companion pencil'' $(\A,\B)$, where the only requirement is that $\det \P(z) = K\det\left( z\B-\A\right)$ for some nonzero constant $K$.  A companion pencil, therefore, has the same eigenvalues as the matrix polynomial.  Companion pencils that are not linearizations do not necessarily preserve eigenvectors or elementary divisors, and are less useful than linearizations.

The usual \textsl{reversal}\footnote{This definition, which is standard, is particularly appropriate for the monomial basis.  The coefficients of the reversed matrix polynomial in the monomial basis are simply the same matrices in reverse order.  The notion of a reversal, however, is independent of the basis used, and indeed reversals can be done differently.  In~\cite{corless2007pseudospectra} for instance we find a slightly different definition of reversal, appropriate for computation in a Lagrange or Hermite interpolational basis, which maps an arbitrary finite point to infinity; this difference allows for greater numerical stability.
} of a matrix polynomial of grade~$\ell$ is the polynomial $\mathrm{rev}\, \P(z) = z^\ell \P(z^{-1})$.  A linearization $\L(z) = z \C_1 - \C_0$ of $\P$ is called a \textsl{strong} linearization if $\mathrm{rev}\,\L(z) = \mat{C}_1-z\mat{C}_0$ is also a linearization of $\mathrm{rev}\,\P(z)$.

If $\L(z) = z \C_{1} - \C_{0} \in \mathbb{C}^{N \times N}[z]$ is a  {linearization} of $\P(z)$ in $z \in \Sigma$, then as a necessary consequence $\mathrm{det}(\P(z)) = q(z)\mathrm{det}(\L(z)) = \mathrm{det}(z \C_{1} - \C_{0})$ for some rational $q(z)$ whose zeros and poles lie outside $\Sigma$. The eigenvalues of $\P$ in $\Sigma$ are thus computable from the generalized eigenvalues of $\L$.  For this linearized problem several standard methods are available.


A \textsl{standard pair} $(\X,\T)$ for a regular matrix polynomial $\P(z)$ expressed in the monomial basis with coefficients $\P_k$ is defined in~\cite{gohberg2009matrix} or in~\cite{Liesen:HLA:2013} as having the following properties: $\X$ has dimension $n \times n\ell$, $\T$ has dimension $n\ell \times n\ell$, 
\begin{equation}
    \sum_{k=0}^\ell \P_k \X \T^k = 0\>,
\end{equation}
and that the $n\ell$ by $n\ell$ matrix
\begin{equation}
    \Q = \begin{bmatrix}
    \X \\
    \X \T \\
    \vdots \\
    \X \T^{\ell-1}
    \end{bmatrix}\label{eq:stacktriple}
\end{equation}
is nonsingular. We can then define a third matrix 
\begin{equation}
    \Y = \Q^{-1}\begin{bmatrix} \mat{0}_n \\ \vdots \\
    \mat{0}_n\\
    \I_n
    \end{bmatrix}
\end{equation}
and say that the triple $(\X,\T,\Y)$ is a \textsl{standard triple} for a \textsl{monic} $\P(z)$.
It is pointed out in~\cite{Liesen:HLA:2013} that monicity of $\P(z)$ is not required for many of the formul\ae\ to do with standard pairs (but is required for some).

Theorem 12.1.4 of~\cite{gohberg2005indefinite} states that 
if there are matrices $\X$, $\T$, and $\Y$ of dimension $n \times n\ell$, $n\ell \times n\ell$, and $n\ell \times n$ for which 
\begin{equation}
    \P^{-1}(z) = \X(z \I_n- \T)^{-1}\Y
\end{equation}
then $(\X, \T, \Y)$ is a standard triple for $\P(z)$. This is also reported in Lemma~2 in~\cite{al2012standard}. 
There are two other representations of a matrix polynomial given a standard triple: the right canonical form, and the left canonical form.  See Theorem 2.4 in~\cite{gohberg2009matrix}. However, we do not need those representations for algebraic linearization: it is the resolvent form above that we seek to generalize in this paper.  The reason is that it is this formula that is used in the proof that algebraic linearizations can be performed when the component matrix polynomials are expressed in different bases.

A polynomial basis $\{\phi_k(z)\}_{k=0}^\ell$ for polynomials of grade~$\ell$ is a set of polynomials for which there is a nonsingular matrix $\Ph$ relating the polynomials $\phi_k(z)$ to the monomials $1$, $z$, $\ldots$, $z^{\ell}$.  We can write this as 
\begin{equation}
    \begin{bmatrix}
    \phi_\ell(z)\\
    \phi_{\ell-1}(z)\\
    \vdots\\
    \phi_0(z)
    \end{bmatrix}
    = \Ph \begin{bmatrix}
    z^{\ell}\\
    z^{\ell-1}\\
    z^{\ell-2}\\
    \vdots\\
    1
    \end{bmatrix}\>.
\end{equation}
Frequently, we want the $\ell \times \ell$ matrix that only goes up to grade $\ell-1$.  This should not cause confusion.  This matrix forms the foundation for the proofs in~\cite{corless2021equivalence} of strict equivalence of various linearizations.
The matrix $\Ph$ is called the change-of-basis matrix and is usually exponentially ill-conditioned in the dimension.  For example, for the Bernstein polynomials $\phi_j^{\ell}(z) = \binom{\ell}{j}z^j(1-z)^{(\ell-j)}$ the change-of-basis matrix has entries $\phi_{i,j} = \binom{j}{i}/\binom{\ell}{j}$ and condition number $K_1 = K_\infty = (\ell+1)\binom{\ell}{s}2^{\ell-s}$ where $s = \lceil (\ell-2)/3\rceil$ (the notation $\lceil x \rceil$ means the ceiling of $x$, the least integer not smaller than $x$). A short computation shows $K \sim 3^{\ell+1}\sqrt{\ell/4{\pi}} $ as $\ell \to \infty$ and is thus exponentially growing with the dimension.

The constructions and definitions of standard triple discussed above are apparently tied to the monomial basis because of the powers $\mat{T}^k$ in Equation~\eqref{eq:stacktriple}. We would like to relax this restriction and extend the notion of standard triple to other bases, and also to the non-monic case.
In particular, we would like the following extension of Theorem 2.4 in~\cite{gohberg2009matrix} or Theorem 12.1.4 in~\cite{gohberg2005indefinite} to be available: If a matrix $\X \in \mathbb{C}^{n \times N}$, the linearization $\L(z)= z \C_{1} - \C_{0} \in \mathbb{C}^{N\times N}[z]$, and a matrix $\Y \in \mathbb{C}^{N \times n}$ satisfy
\begin{equation}
	\P^{-1}(z) = \X(z \C_{1} - \C_{0})^{-1}\Y \label{eq:resolventrepresentation}
\end{equation}
for $z \notin z(\P)$ (the set of polynomial eigenvalues of $\P$), then $\X$, $\L(z)$, and $\Y$ form a \textsl{generalized standard triple} for $\P(z)$. This obviously requires regularity of $\P$ because the formula contains $\P^{-1}(z)$.  

Indeed, we simply require $\L(z)$ to be a linearization (or local linearization, restricted to some set $\Sigma$) and take this extension as a \textsl{definition}. Such things exist, as we will demonstrate, and are useful, as shown in~\cite{chan2018algebraic}.

A referee pointed out that a simple characterization of linearizations $\L(z)$ can be found in~\cite{Fassbender2017}; this can be used as a starting point, here.

\begin{Definition}\label{def:GST}
Matrices $\mat{X}$, $z\mat{C}_1-\mat{C}_0$, and $\mat{Y}$ form a \textsl{generalized standard triple} for the regular matrix polynomial $\mat{P}(z)$ if $\L(z) = z \C_{1} - \C_{0}$ is a linearization of $\mat{P}$ and  Equation~\eqref{eq:resolventrepresentation} holds.
\end{Definition}

Note that the matrices $\X$ and $\Y$ do not depend on~$z$, but the linearization $\L(z)$ does, albeit only linearly; we could instead have chosen to use the words ``standard quadruple'' to mean $(\X,\C_1,\C_0,\Y)$ where $z$ does not appear of any of these matrices, but this quibble seems to be a matter of aesthetics only; we may use the term ``triple'' to refer to $\X$, the linearization, and $\Y$.  

\begin{Proposition}\label{rem:readoff}
If $\L(z)$ is a linearization for $\P(z)$, then there exists $\X$ and $\Y$ forming a generalized standard triple with $\L(z)$ in the above sense.
\end{Proposition}
\begin{proof}
This is already proved in~\cite{gohberg2009matrix}. If $\L(z)$ is a linearization for $\P(z)$ then there exist unimodular matrix polynomials $\E(z)$ and $\F(z)$ with $\F^{-1}(z)(z\C_1-\C_0)^{-1}\E^{-1}(z) = \diag( \P^{-1}(z), \I_n, \ldots, \I_n)$.  By premultiplying by $\X_p = [\I_n, 0, \ldots, 0]$ and postmultiplying by $\Y_p = [\I_n, 0, \ldots, 0]^T$, we may find $\X = \X_p\F^{-1}(z)$ and $\Y = \E^{-1}(z)\Y_p$ so that Equation~\eqref{eq:resolventrepresentation} holds. 
\end{proof}
\begin{remark}
Thus, as a referee pointed out, generalized standard triples may be read off from the proof that $\L(z)$ is indeed a linearization (we shall see this explicitly, shortly).
The matrices $\E(z)$ and $\F(z)$ above are not unique; given a nonsingular constant matrix $\mat{B}$ of the right dimension, $\diag{(\I_n,\mat{B})}\E(z) (z\C_1 - \C_0)\F(z)\diag{(\I_n,\mat{B}^{-1})}$ is also $\diag(\P(z),\I_n, \ldots, \I_n)$.  Therefore, generalized standard triples are also not unique.
\end{remark}

Without using the rank criterion characterization of~\cite{Fassbender2017}, say, it is not a priori clear that if Equation~\eqref{eq:resolventrepresentation} holds then $\L(z)$ is necessarily a linearization of $\P(z)$.  We do have in any case, however, the following:

\begin{lemma}
If $\X$, $\Y$, and the matrix pencil $\L(z)=z\mat{C}_1-\mat{C}_0$ are such that Equation~\eqref{eq:resolventrepresentation} holds, then the matrix pencil $L(z)$ is at least a companion pencil for $\mat{P}(z)$.
\end{lemma}
\begin{proof}
The norm of the resolvent $\|\mat{P}^{-1}(z)\|$ will be large if and only if $\| \left(z\mat{C}_1-\mat{C}_0\right)^{-1}\|$ is large. 
\end{proof}
For various reasons, we usually do not wish to invert a ``leading coefficient'' here; for instance, if the polynomial basis is not degree-graded, e.g.~for the Bernstein basis,
then in order to even look at the true leading coefficient, 
we have to form a particular linear combination of the existing coefficients.  In floating-point arithmetic, rounding errors can disguise the rank of the resulting matrix, hence our interest in the generalization.

If $\X$, $z\C_1-\C_0$, and $\Y$ form a generalized standard triple according to our definition, then so do $\X\U$, $\V^{-1}(z\C_1-\C_0)\U$, and $\V^{-1}\Y$ for any nonsingular matrices $\U$ and $\V$ of dimension $N$ by $N$.  

Several similarities are used very frequently.  For convenience, we describe two of the most common explicitly here.
\begin{lemma}[Flipping]
	Put $\J \text{as the} \; N \times N$ ``anti-identity'', also called the \textsl{sip} matrix, for \textsl{standard involutory permutation},  $\J_{i,j} = 0$ unless $i + j = N + 1$ when $\J_{i, N+1-i}~=~1$. Then $\J^{2}~=~\I$ and the ``flipped'' linearization $\mat{L_{F}}(z) = \J(z \C_{1} - \C_{0})\J$ has in its generalized standard triple the matrices $\mat{X_{F}} = \mat{XJ}$ and $\mat{Y_{F}} = \mat{JY}$. 
\end{lemma}

\begin{proof}
	Immediate.
\end{proof}

\begin{remark}
	Flipping switches both the order of the equations and the order of the variables. It obviously does not change eigenvalues. Flipping, transposition, and flipping-with-transposition give four common equivalent linearizations~\cite{taussky1959similarity}.
\end{remark}

\subsection{Algebraic Linearizations \label{sec:AlgebraicLinearization}}
An \textsl{algebraic linearization} $(\H,\D_H)$, as referred to in the title of this present note, is defined in \cite{chan2018algebraic} as a linearization
$(\H,\D_H)$ of a matrix polynomial $\mat{h}(z) = z\mat{a}(z)\mat{d}_0\mat{b}(z) + \mat{c}$ constructed recursively from linearizations $(\A,\mat{D}_A)$ and $(\B,\mat{D}_B)$ of the lower-grade component matrix polynomials $\mat{a}(z)$ and $\mat{b}(z)$, together with constant matrices $\mat{d}_0$ and $\mat{c}$. The paper~\cite{chan2018algebraic} did not give an explicit unimodular pair $(\E_H(z),\F_H(z))$ that reduces the linearization to diag$(z\mat{a}(z)\mat{d}_0\mat{b}(z)+\mat{c}_0, \I_{N-n})$, proving that the construction actually gave a linearization, so we give a method to construct them here.
Without loss of generality we take $\mat{d}_0 = \I_n$.
\begin{theorem}
If the $n$ by $n$ matrix polynomial $\mat{a}(z)$ has local linearization $(\A,\mat{D}_A)$ on $\Sigma_{\A}$ with nonsingular pair
$(\E_A(z), \F_A(z))$ and if the $n$ by $n$ matrix polynomial $\mat{b}(z)$ has local linearization $(\B,\mat{D}_B)$ on $\Sigma_{\B}$ with nonsingular pair
$(\E_B(z), \F_B(z))$
then the pencil $z\D_H-\H$ is a local linearization of $\mat{h}(z) = z\mat{a}(z)\mat{b}(z) + \C$ on $\Sigma_{\A} \cap \Sigma_{\B}$, where the matrices $\D_H$ and $\H$ are given as follows:
\begin{equation}
    	\D_H =
	\left[
		\begin{array}{ccc}
			\D_A & & \\
			& \I_n & \\
			& & \D_B
		\end{array}
	\right]
\end{equation}
and
\begin{equation}
\H =	\left[
		\begin{array}{ccc}
			\A & \mat{0}_{N_A,n} & -\mat{Y}_{A}\mat{c}\mat{X}_{B} \\
			-\mat{X}_{A} & \mat{0}_n & \mat{0}_{n,N_B} \\
			\mat{0}_{N_B,N_A} & -\mat{Y}_{B} & \mat{B}
		\end{array}
	\right]\>.
\end{equation}
Here $\X_A = [\I_n, 0, \ldots, 0]\F_A^{-1}(z)$
$\Y_A = \E_A^{-1}(z)[\I_n, 0, \ldots, 0]^T$ and likewise $\X_B =[\I_n, 0, \ldots, 0]\F_B^{-1}(z)$ and $\Y_B=\E_A^{-1}(z)[\I_n, 0, \ldots, 0]^T$ give the elements of the (generalized) standard triples for $\mat{a}(z)$ and $\mat{b}(z)$.
\end{theorem}
\begin{proof}
We first construct
\begin{equation}
    \E_1(z) = \begin{bmatrix}
        \E_A(z) & & \\
             &\I_n & \\
             &     &\E_B(z)
    \end{bmatrix}
\end{equation}
and
\begin{equation}
    \F_1(z) = \begin{bmatrix}
        \F_A(z) & & \\
             &\I_n & \\
             &     &\F_B(z)
    \end{bmatrix}\>.
\end{equation}
Applying them we get
\begin{equation}
\E_1(z)(z\D_H - \H)\F_1(z)    = \begin{bmatrix}
\mat{a}(z) & & & \E_A(z)\Y_A \mat{c} \X_B\F_B(z) & \\
      & \I_{N_A-n}& & & & & \\
\X_A\F_A(z) & & z\I_n& & \\
      &           & \E_B(z)\Y_B & \mat{b}(z) & \\
      &           &          &       & \I_{N_B-n}
    \end{bmatrix}\>.
\end{equation}
Simplifying and using the definitions of the matrices appearing in the standard triples, we get
\begin{equation}
\left[
    \begin{array}{cc|c|cc}
        \mat{a}(z) & & & \mat{c} & \\
        & \I_{N_A-n}& & &  \\
        \hline
        \I_n & & z\I_n& & \\
        \hline
        & & \I_n & \mat{b}(z) & \\
        & & & & \I_{N_B-n}
    \end{array}
\right]\>,
\end{equation}
which is permutationally equivalent to
\begin{equation}
    \left[
    \begin{array}{cc|cc|c}
        \mat{a}(z) & & \mat{c} & & \\
        & \I_{N_A-n} & & & \\
        \hline
        & & \mat{b}(z) & & \I_n \\
        & & & \I_{N_B - n} & \\
        \hline
        \I_n & & & &z\I_n
    \end{array}
    \right] \>.
\end{equation}
Adding to the third block column, the last block column multiplied by $-\mat{b}(z)$, we get
\begin{equation}
    \left[
    \begin{array}{cc|cc|c}
        \mat{a}(z) & & \mat{c} & & \\
        & \I_{N_A - n} & & & \\
        \hline
        & & 0 & & \I_n \\
        & & & \I_{N_B - n} & \\
        \hline
        \I_n & & -z\mat{b}(z) & & z\I_n
    \end{array}
    \right] \>,
\end{equation}
which is permutationally equivalent to
\begin{equation}
    \left[
    \begin{array}{cc|cc|c}
        \mat{a}(z) & \mat{c} & & & \\
        \I_n & -z\mat{b}(z) & & & z\I_n \\
        \hline
        & & \I_{N_A - n} & & \\
        & & & \I_{N_B - n} & \\
        \hline
        & & & & \I_n
    \end{array}
    \right] \>.
\end{equation}
Adding to the second block row, the last block row multiplied by $-z$, we get
\begin{equation}
    \left[
    \begin{array}{cc|cc|c}
        \mat{a}(z) & \mat{c} & & & \\
        \I_n & -z\mat{b}(z) & & & \\
        \hline
        & & \I_{N_A - n} & & \\
        & & & \I_{N_B - n} & \\
        \hline
        & & & & \I_n
    \end{array}
    \right] \>.
\end{equation}
Finally, take into account that
\begin{equation}
    \begin{bmatrix}
        \I_n & -\mat{a}(z) \\
        & \I_n
    \end{bmatrix}
    \begin{bmatrix}
        \mat{a}(z) & \mat{c} \\
        \I_n & -z\mat{b}(z)
    \end{bmatrix}
    \begin{bmatrix}
        z\mat{b}(z) & \I_n \\
        \I_n &
    \end{bmatrix}
    =
    \begin{bmatrix}
        z\mat{a}(z)\mat{b}(z) + \mat{c} & \\
        & \I_n
    \end{bmatrix} \>.
\end{equation}
This completes the proof.
\end{proof}
\begin{remark}
In the case that $\Sigma_{\A} = \Sigma_{\A}=\fieldF$, then the intersection is also $\fieldF$, and this establishes that ``algebraic linearizations'' are linearizations if their components are linearizations.
\end{remark}

\begin{remark}
    The generalized standard triple for the algebraic linearization given here has $\X_H = [0, 0, \X_B]$ and $\Y_H = [\Y_A^T, 0, 0]^T$. For these linearizations, there is no notion of expressing $1$ as a linear combination of anything, because this formulation is independent of particular polynomial bases, and indeed may use different bases for different submatrices.
\end{remark}

Algebraic linearizations offer a new class of linearizations.  In~\cite{chan2018algebraic} examples are given where the eigenvalue conditioning of such linearizations is \emph{exponentially} better than that of the Frobenius linearization; this demonstrates that this class of linearizations potentially offers  more numerically stable algorithms\footnote{Any algorithm that transforms a well-conditioned problem into an ill-conditioned one as a step along the way is likely to be a numerically unstable algorithm.} for computing matrix polynomial eigenvalues than standard linearizations do. The extent to which this is possible in general has not yet been explored: at this point, we only know that this \emph{can} happen for some cases.

The recursive construction of algebraic linearizations relies on generalized standard triples of each of the component matrix polynomials, and (as does the unrelated  paper~\cite{robol2017framework}) allows different polynomial bases to be used for each component. This present note provides some explicit formulas for generalized standard triples in various bases, for reference. As one reviewer points out, these formulas \textsl{could} simply be obtained by reading the proofs that these linearizations are indeed linearizations; one purpose of this paper is simply convenience.

\section{Expressing $1$ in the basis gives the triple\label{sec:universal}}
If the $\phi_k(z)$, $0 \le k \le \ell-1$ form a basis, we may express the polynomial $1$ in that basis: then  $1=\sum_{k=0}^{\ell-1}e_k\phi_k(z)$ defines the coefficients $e_k$ uniquely. Putting
\begin{equation}
	\X = 
    \begin{bmatrix}
		e_{\ell-1} & e_{\ell-2} & \cdots & e_{1} & e_{0}
	\end{bmatrix}     
\otimes \I_n
\end{equation}  
for an appropriate choice of basis always gives our generalized standard triple $\P^{-1}(z)=\X(z\C_1-\C_0)^{-1}\Y$ with 
\begin{equation}
\Y=
\begin{bmatrix}
	\I_n& \mat{0}_n & \mat{0}_n & \cdots & \mat{0}_n
\end{bmatrix}^T = \mat{e}_1 \otimes \I_n \>.
\end{equation}
[Here the notation~$\mat{e}_1$ means the first elementary vector: although printed in bold type, it does not look very different from the scalar $e_1$, not bold, printed earlier as a component of $\X$.]
We prove this below for all the elementary linearizations we use in this paper.  Moreover, if we replace $\Y$ above with $\mat{v} \otimes \I_n$ for a general anszatz column vector $\mat{v}$, then the theorem is also true for all the linearizations of~\cite{Fassbender2017} as well.
\begin{theorem}
Let $\P(z)$ be a regular matrix polynomial. Consider a linearization $\L(z)$ of $\P(z)$ such that
\begin{equation}
    \L(z)\left(\Ph_\ell(z) \otimes \I_m\right) = \left(e_1 \otimes \I_m\right)\P(z) \>,
\end{equation}
where $e_1 = \begin{bmatrix}1 & 0 & \cdots & 0\end{bmatrix}^T \in \mathbb{C}^{\ell}$ and $\Ph_{\ell}(z) = \begin{bmatrix}\phi_{\ell - 1}(z) & \cdots & \phi_0(z)\end{bmatrix}^T$. Let $x$ be a vector such that $x\Ph_\ell(z) = 1$ and define $\X = x \otimes \I_m$ and $\Y = e_1 \otimes \I_m$, then
\begin{equation}
    \P(z)^{-1} = \X\L(z)^{-1}\Y \>.
\end{equation}
\label{thm:maintheorem}
\end{theorem}

\begin{proof}
The following proof, which uses an idea of an anonymous referee, is simpler than our original one.  For each of the polynomial bases we examine in this paper, the linearization satisfies either
\begin{equation}
\L(z)\begin{bmatrix}
\phi_{\ell-1}(z)\mat{I}_n\\
\phi_{\ell-2}(z)\mat{I}_n\\
\vdots\\
\phi_{0}(z)\mat{I}_n
\end{bmatrix}
= \begin{bmatrix}
\mat{I}_n \\
\mat{0}_n \\
\vdots\\
\mat{0}_n
\end{bmatrix}\mat{P}(z)\>,
\end{equation}
for degree-graded bases, or similar statements for Bernstein bases and Lagrange and Hermite interpolational bases, as follows.  For the Bernstein basis, the
polynomial elements in the vector on the left are multiples of $B^{\ell-1}_{j}(z)$: 
$[\ell/1\cdot B^{\ell-1}_{\ell-1}(z), \ell/2\cdot B^{\ell-1}_{\ell-1}(z), \ldots, \ell/\ell\cdot B^{\ell-1}_{0}(z)]^T$. For the Lagrange basis, the vector on the left is $[w(z), \ell_0(z), \ell_1(z), \ldots, \ell_{\ell}(z)]^T$. For the Hermite interpolational basis\footnote{The paper~\cite{corless2021equivalence} did not prove that the Hermite companion pencil discussed here is in fact a linearization. We believe that it is, but there is no proof published yet.}, it is the same as for the Lagrange but with the Lagrange basis elements replaced with the Hermite interpolational basis elements.

This offers an explicit slight extension of the results of~\cite{Fassbender2017}, and their use of a general \emph{anszatz} vector $\mat{v}$ generalizes our results, in the following way:
They consider the set of all \emph{matrix pencils} $\L(z)$ satisfying
\begin{equation}
     \L(z)\left( \Ph \otimes \I_n\right) = \mat{v} \otimes \P(z)
\end{equation}
and later characterize just which of these matrix pencils are linearizations.  They did not explicitly consider Lagrange or Hermite interpolational basis polynomials, or the Bernstein basis, but as they point out their proofs go through unchanged for matrix pencils where $z$ appears only on the diagonal of the matrix pencil, as it does for all the cases we consider here except the Bernstein case.  Their use of the general anszatz vector (all the examples we have considered just use $\mat{v} = \mat{e}_1$) extends our result to the case $\Y = \mat{v} \otimes \I_n$.

Premultiplying by ${\L}^{-1}(z)$ and post-multiplying by $\mat{P}^{-1}(z)$, we have
\begin{equation}
{\L}^{-1}(z)\begin{bmatrix}
\mat{v} \\
\mat{0}_n \\
\vdots\\
\mat{0}_n
\end{bmatrix}
= \begin{bmatrix}
\phi_{\ell-1}(z)\mat{I}_n\\
\phi_{\ell-2}(z)\mat{I}_n\\
\vdots\\
\phi_{0}(z)\mat{I}_n
\end{bmatrix}\mat{P}^{-1}(z)\>.    
\end{equation}
If $1 = \sum_{k=0}^{\ell-1} e_k \phi_k(z)$ is the expression of $1$ in that basis, then premultiplying both sides by
\begin{equation*}
\X = \begin{bmatrix}
e_{\ell-1}\mat{I}_n & e_{\ell-2}\mat{I}_n 
& \ldots & e_{0}\mat{I}_n
\end{bmatrix}
\end{equation*}
gives the theorem.  Compare also Remark~\ref{rem:readoff} which gives another formula for $\X$ and $\Y$.

Note that in the Bernstein, Lagrange, and Hermite interpolational cases, $1$ can be expressed as a linear combination of the elements given; for Lagrange and Hermite the coefficient of $w(z)$ is $0$.
\end{proof}

\begin{remark}
    There are linearizations not explicitly considered in this paper; for instance, a referee has pointed out that when a matrix polynomial is expressed in a basis where the elements satisfy a linear recurrence, then there is an automatic way to build what is called a CORK linearization.  See~\cite{guttel2017nonlinear} and~\cite{van2015linearization} for details.  
\end{remark}

In what follows we examine specific cases in detail and supply specific proofs for each basis.  Indeed, much of the utility of this paper is simply writing down those details, which will allow easier programming for the uses of these generalized standard triples.

\section{Scalar examples of generalized standard triples} \label{triples-sec}
In this section, we tabulate generalized standard triples for four classes of linearizations. 

In the special case $n = 1$ and when the monomial basis is used, a linearization is usually simplified by dividing by the leading coefficient, making the result monic and the second matrix of the pair just becomes the identity.  The remaining matrix is called a ``companion matrix'' or Frobenius companion\footnote{The \textsl{Frobenius form} of a matrix is related, but different: see for instance~\cite{storjohann1998frobenius}.}. Thus finding roots of a scalar polynomial can be done by finding eigenvalues of the companion matrix. 
Kublanovskaya calls these ``accompanying pencils'' in~\cite{kublanovskaya1999methods}.

Construction of a linearization from a companion matrix is, when possible at all, a simple matter of the Kronecker (tensor) product: given $\C_{1}$, $\C_{0} \in \mathbb{C}^{n \times n}$, take $\widetilde{\C_{1}} = \C_{1} \otimes \I_{n}$ and then replace each block $p_k\I_{n}$ with the corresponding matrix coefficient $\cof{k} \in \mathbb{C}^{r \times r}$ (the first $p_k$, in $p_k \I_{n}$, is the symbolic coefficient from $p(z) = \sum_{k = 0}^{\ell} p_k\phi_{k}(z)$; the matrix coefficient $\cof{k} \in \mathbb{C}^{r \times r}$ is from $\P(z) = \sum_{k = 0}^{\ell}\cof{k}\phi_{k}(z).)$ This will be clearer by example.

\subsection{Bases with three-term recurrence relations} \label{recur-subsec}
The monomial basis, the shifted monomial basis, the Taylor basis, the Newton interpolational bases, and many common orthogonal polynomial bases all have three-term recurrence relations that, except for initial cases, can be written 
\begin{equation}
	z\phi_{k}(z) = \alpha_{k}\phi_{k+1}(z) + \beta_{k}\phi_{k}(z) + \gamma_{k}\phi_{k-1}(z) \>.
\end{equation}
In all cases, we have $\alpha_{k}\ne 0$.
For instance, the Chebyshev polynomial recurrence is usually written $T_{n+1}(z) = 2zT_n(z) - T_{n-1}(z)$ but is easily rewritten in the above form by isolating $zT_n(z)$, and all Chebyshev $\alpha_k = 1/2$ for $k>1$.
We give a selection in Table~\ref{tab:shortlist}, and refer the reader to section 18.9 of the Digital Library of Mathematical Functions (\url{dlmf.nist.gov}) for more. See also \cite{gautschi2016orthogonal}.

\begin{table}
	\centering
    \begin{tabular}{|c|c|c|c|c|c|c|}
    \hline
    	$\phi_{k}(z)$ & Name & $\alpha_{k}$ & $\beta_{k}$ & $\gamma_{k}$ & $\phi_{0}$ & $\phi_{1}$ \\
        \hline
        $z^{k}$ & monomial & $1$ & $0$ & $0$ & $1$ & $z$ \\
        \hline
        $(z - a)^{k}$ & shifted monomial & $1$ & $a$ & $0$ & $1$ & $z - a$ \\
        \hline
        $\sfrac{(z - a)^{k}}{k!}$ & Taylor & $n + 1$ & $a$ & $0$ & $1$ & $z - a$ \\
        \hline
        $\prod_{j=0}^{k-1}(z - \tau_{j})$ & Newton interpolational & $1$ & $\tau_{n}$ & $0$ & $1$ & $z - \tau_{0}$ \\
        \hline
        $T_{k}(z) = \cos\left(k\cos^{-1}(z)\right)$ & Chebyshev & $\sfrac{1}{2}$ & $0$ & $\sfrac{1}{2}$ & $1$ & $z$ \\
        \hline
        $P_{k}(z)$ & Legendre & $\sfrac{(k+1)}{(2k+1)}$ & $0$ & $\sfrac{k}{(2k+1)}$ & $1$ & $z$ \\
        \hline
    \end{tabular}
    \caption{\label{tab:shortlist}A short list of three-term recurrence relations for some important polynomial bases discussed in Section~\ref{recur-subsec}. For a more comprehensive list, see The Digital Library of Mathematical Functions. These relations and others are coded in Walter Gautschi's packages OPQ and SOPQ~\cite{gautschi2016orthogonal} and in the \texttt{MatrixPolynomialObject} implementation package in Maple (see~\cite{Jeffrey:HLA:2013}).}
\end{table}
For all such bases, we have the linearization\footnote{For exposition, we follow Peter Lancaster's dictum, namely that the $5\times 5$ case almost always gives the idea.}
\begin{equation}
    z\C_1 - \C_0 =
    z\left[\begin{array}{c|cccc}
    	\dfrac{\scof{5}}{\alpha_{4}} & & & & \\
    \hline
        & 1 & & & \\
        & & 1 & & \\
        & & & 1 & \\
        & & & & 1
    \end{array}\right] -
    \left[\begin{array}{c|cccc}
    	-\scof{4} + \dfrac{\beta_{4}}{\alpha_{4}}\scof{5} & -\scof{3} + \dfrac{\gamma_{4}}{\alpha_{4}}\scof{5} & -\scof{2} & -\scof{1} & -\scof{0} \\
    	\hline
        \alpha_{3} & \beta_{3} & \gamma_{3} & & \\
        & \alpha_{2} & \beta_{2} & \gamma_{2} & \\
        & & \alpha_{1} & \beta_{1} & \gamma_{1} \\
        & & & \alpha_{0} & \beta_{0}
    \end{array}\right] \>,
\end{equation}
(remember that $\alpha_k \ne 0$) and
\begin{align}
	\X &= 
    \begin{bmatrix}
    	0 & 0 & 0 & 0 & 1
    \end{bmatrix} \>, \\
    \Y &=
    \begin{bmatrix}
    	1 & 0 & 0 & 0 & 0
    \end{bmatrix}^{\mathrm{T}} \>.
\end{align}

For instance, a Newton interpolational basis on the nodes $\tau_{0}$, $\tau_{1}$, $\ldots$, $\tau_{5}$ has the linearization for matrix polynomials of this grade,
\begin{equation}
	z
    \begin{bmatrix}
    	\P_5 & & & & \\
        & \I_{n} & & & \\
        & & \I_{n} & & \\
        & & & \I_{n} & \\
        & & & & \I_{n}
    \end{bmatrix}
    -
    \begin{bmatrix}
    	-\P_{4} + \tau_{4}\P_{5} & -\P_{3} & -\P_{2} & -\P_{1} & -\cof{0} \\
        \I_{n} & \tau_{3}\I_{n} & & & \\
        & \I_{n} & \tau_{2}\I_{n} & & \\
        & & \I_{n} & \tau_{1}\I_{n} & \\
        & & & \I_{n} & \tau_{0}\I_{n}
    \end{bmatrix} \>.
\end{equation}

\subsection{The Bernstein basis} \label{bernstein-subsec}
The set of polynomials $\{B_{k}^{\ell}(z)\}_{k=0}^{\ell} = \binom{\ell}{k}z^k(1-z)^{\ell-k}$ is a set of $\ell+1$ polynomials each of exact degree $\ell$ that together forms a basis for polynomials of grade~$\ell$. 
Bernstein polynomimals have many applications, for example in Computer Aided Geometric Design (CAGD), and many important properties including that of optimal condition number over all bases positive on $\left[0, 1\right]$. They do not satisfy a simple three term recurrence relation of the form discussed in Section~\ref{recur-subsec}, although they satisfy an interesting and useful ``degree-elevation'' recurrence, namely
\begin{equation}
     (j+1)B_{j+1}^n(z) + (n-j)B_{j}^n(z) = n B_{j}^{n-1}(z)\>,
\end{equation}
which specifically demonstrates that a sum of Bernstein polynomials of degree~$n$ might actually have degree strictly less than~$n$.
See~\cite{farouki2012bernstein},~\cite{farouki1996optimal}, and~\cite{farouki1987numerical} for more details of Bernstein bases.

A Bernstein companion pencil for $p_{5}(z) = \sum_{k=0}^{5}\scof{k}B_{k}^{5}(z)$ is
\begin{align}
	z\C_{1} - \C_0 &=
    z\begin{bmatrix}
    	 - \scof{4}+\dfrac{1}{5}\scof{5} & -\scof{3} & -\scof{2} & -\scof{1} & -\scof{0} \\
        1 & \dfrac{2}{4} & & & \\
        & 1 & \dfrac{3}{3} & & \\
        & & 1 & \dfrac{4}{2} & \\
        & & & 1 & \dfrac{5}{1}
    \end{bmatrix} -
    \begin{bmatrix}
    	-\scof{4} & -\scof{3} & -\scof{2} & -\scof{1} & -\scof{0} \\
        1 & 0 & & & \\
        & 1 & 0 & & \\
        & & 1 & 0 & \\
        & & & 1 & 0
    \end{bmatrix} \>, \\
    \X &= 
    \begin{bmatrix}
    	\dfrac{1}{5} & \dfrac{2}{5} & \dfrac{3}{5} & \dfrac{4}{5} & \dfrac{5}{5}
    \end{bmatrix} \>, \\
    \Y &=
    \begin{bmatrix}
    	1 & 0 & 0 & 0 & 0
    \end{bmatrix}^{\mathrm{T}} \>.
\end{align}
For a construction of rational $\E(z)$ and $\F(z)$ that show this is a local linearization (unless $z=1$ is an eigenvalue), see~\cite{amiraslani2008linearization}. The paper~\cite{mackey2016linearizations} goes further and gives a general method of constructing all strong linearizations for Bernstein matrix polynomials, including this one (which they show is strictly equivalent to the second companion form for the monomial basis, without any restrictions on $z$ or on the leading coefficient). 
We have $p^{-1}(z) = \X(z \C_{1} - \C_{0})^{-1}\Y$ if $p(z) \neq 0$. This pencil was first analyzed in~\cite{jonsson2001eigenvalue} and~\cite{jonsson2004solving}. One of the present authors independently invented and implemented a version of this companion pencil in Maple (except using $\P^{\mathrm{T}}(z)$, and flipped from the above form) in about $2004$. 
For a proof of numerical stability, see the original thesis~\cite{jonsson2001eigenvalue}.  
The standard triple is, we believe, new to this paper.

\begin{example}[Singular leading coefficient case]
\begin{align} \label{eqn:BernsteinNonMonicExample}
\P(z) &= 
\begin{bmatrix}
	\sfrac{29}{100}  & -\sfrac{8}{25} \\
	\sfrac{7}{10} & -\sfrac{1}{100}
\end{bmatrix}
B^{3}_{0}(z) +
\begin{bmatrix}
	-\sfrac{41}{50}  & \sfrac{41}{100} \\
	-\sfrac{7}{10} & \sfrac{91}{100}
\end{bmatrix}
B^{3}_{1}(z) \\ \nonumber &+
\begin{bmatrix}
	\sfrac{9}{10}  & \sfrac{19}{100} \\
	\sfrac{4}{5} & \sfrac{22}{25}
\end{bmatrix}
B^{3}_{2}(z) +
\begin{bmatrix}
    1  & 1 \\
	\sfrac{9851}{1980} & 0
\end{bmatrix}
B^{3}_{3}(z) \>.
\end{align}
Expressing Equation~\eqref{eqn:BernsteinNonMonicExample} into the monomial basis, we have
\begin{equation*}
    \P(z) = 
    \begin{bmatrix}
        {\sfrac{29}{100}}&-{\sfrac{8}{25}}\\
        {\sfrac{7}{10}}&-{\sfrac{1}{100}}
    \end{bmatrix} +
    \begin{bmatrix}
        {\sfrac{29}{100}}&-{\sfrac{8}{25}}\\
        {\sfrac{7}{10}}&-{\sfrac{1}{100}}
    \end{bmatrix} z +
    \begin{bmatrix}
        {\sfrac{849}{100}}&-{\sfrac{57}{20}}\\
        {\sfrac{87}{10}}&-{\sfrac{57}{20}}
    \end{bmatrix} z^2 +
    \begin{bmatrix}
        -{\sfrac{89}{20}}&{\sfrac{99}{50}}\\
        -{\sfrac{89}{396}}&\sfrac{1}{10}
    \end{bmatrix} z^3 \>.
\end{equation*}
Taking the determinant of the leading coefficient
\begin{equation*}
    \det\left(
    \begin{bmatrix}
        -{\sfrac{89}{20}}&{\sfrac{99}{50}}\\
        -{\sfrac{89}{396}}&\sfrac{1}{10}
    \end{bmatrix}
    \right)
    = \left(-{\sfrac{89}{20}}\right) \left(\sfrac{1}{10}\right) - \left({\sfrac{99}{50}}\right) \left(-{\sfrac{89}{396}}\right) = 0 \>,
\end{equation*}
we can observe that leading coefficient is singular, and thus, this matrix polynomial is non-monic. The standard triple for Equation~\eqref{eqn:BernsteinNonMonicExample} is
\begin{gather*}
	\C_{0} =
	 \begin{bmatrix}
	 	-{\sfrac{9}{10}}&-{\sfrac{19}{100}}&{\sfrac{41}{50}}&-{\sfrac{41}{100}}&-{\sfrac{29}{100}}&{\sfrac{8}{25}}\\
	 	-\sfrac{4}{5}&-{\sfrac{22}{25}}&{\sfrac{7}{10}}&-{\sfrac{91}{100}}&-{\sfrac{7}{10}}&{\sfrac{1}{100}}\\
	 	1&0&0&0&0&0\\
	 	0&1&0&0&0&0\\
	 	0&0&1&0&0&0\\ 
	 	0&0&0&1&0&0
	\end{bmatrix} \quad
	\C_{1} =
	\begin{bmatrix}
		-{\sfrac{17}{30}}&{\sfrac{43}{300}}&{\sfrac{41}{50}}&-{\sfrac{41}{100}}&-{\sfrac{29}{100}}&{\sfrac{8}{25}}\\
		{\sfrac{5099}{5940}}&-{\sfrac{22}{25}}&{\sfrac{7}{10}}&-{\sfrac{91}{100}}&-{\sfrac{7}{10}}&{\sfrac{1}{100}}\\
		1&0&1&0&0&0\\ 
		0&1&0&1&0&0\\
		0&0&1&0&3&0\\ 
		0&0&0&1&0&3
	\end{bmatrix} \\
	\X = 
	\begin{bmatrix}
		\sfrac{1}{3}&0&\sfrac{2}{3}&0&1&0\\
		0&\sfrac{1}{3}&0&\sfrac{2}{3}&0&1
	\end{bmatrix} \quad
	\Y =
	\begin{bmatrix}
		1&0\\ 
		0&1\\ 
		0&0\\ 
		0&0\\ 
		0&0\\ 
		0&0
	\end{bmatrix} \>.
\end{gather*}
Then, 
\begin{equation*}
	\X\left(z\C_{1} - \C_{0}\right)^{-1}\Y\P(z) = \I_{2} \>.
\end{equation*}

\end{example}

\subsection{The Lagrange interpolational basis} \label{lagrange-subsec}
There are by now several Lagrange basis linearizations. The use of barycentric forms means that Lagrange interpolation is efficient and numerically stable and is increasing in popularity~\cite{trefethen}. Here is the definition of the first barycentric form for interpolation of polynomials of grade~$\ell$ on the $\ell+1$ distinct nodes $\tau_k \in \mathbb{C}$, $0 \le k \le \ell$.  Take the partial fraction decomposition of the reciprocal of the node polynomial
\begin{equation}
    w(z) = \prod_{k=0}^{\ell} (z - \tau_k)\>,
\end{equation}
namely
\begin{equation}
    \frac{1}{w(z)} = \sum_{k=0}^{\ell} \frac{\beta_k}{z-\tau_k}
\end{equation}
where the coefficients $\beta_k$ occurring in the partial 
fraction decomposition are called the \emph{barycentric weights}. A well-known explicit formula for the $\beta_k$ is
\begin{equation}
	\beta_{k} = \prod_{\substack{j = 0 \\ j \neq k}}^{\ell}(\tau_{k} - \tau_{j})^{-1} \>.
\end{equation}
The Lagrange basis polynomials are normally written
\begin{equation}
    \ell_k(z) = \beta_k \prod_{\substack{j = 0 \\ j \neq k}}^{\ell}(z - \tau_{j}) \>.
\end{equation}

For many sets of nodes (Chebyshev nodes on $[-1, 1]$, or roots of unity on the unit disk), the resulting interpolant is also well-conditioned, and can even be ``better than optimal''~\cite{corless2004bernstein}, see also \cite{carnicer2017optimal}. The pencil we use here is ``too large'' and has (numerically harmless in our experience) spurious roots at infinity\footnote{This numerical harmlessness needs some explanation. In brief, Lagrange basis matrix polynomial eigenvalues will be well-conditioned only in a compact region determined by the interpolation nodes, and are increasingly ill-conditioned towards infinity; in practice this means only small changes in the data are needed to perturb large finite ill-conditioned eigenvalues out to infinity.  Any eigenvalues produced numerically that are well outside the region determined by the interpolation nodes are likely easily perturbed all the way to infinity, and can be safely ignored.}; for alternative formulations, see \cite{van2015linearization}, \cite{nakatsukasa2017vector}. Then, the pencil is $z \C_{1} - \C_{0}$ where
\begin{equation}
	z\C_{1} - \C_{0} =
    z\begin{bmatrix}
    	0 & & & & \\
        & 1 & & & \\
        & & 1 & & \\
        & & & 1 & \\
        & & & & 1
    \end{bmatrix} -
    \begin{bmatrix}
    	0 & -\rho_{0} & -\rho_{1} & -\rho_{2} & -\rho_{3} & -\rho_{4} \\
        \beta_{0} & \tau_{0} & & & & \\
        \beta_{1} & & \tau_{1} & & & \\
        \beta_{2} & & & \tau_{2} & &\\
        \beta_{3} & & & & \tau_{3} & \\
        \beta_{4} & & & & & \tau_{4}
    \end{bmatrix} \>.
\end{equation}
Rational matrices $\E$ and $\F$ demonstrating that this is indeed a local linearization, except when the nodes are eigenvalues, can be found in the appendix of~\cite{amiraslani2008linearization}. 
More is true: by a block permutation argument, one can find constant matrices $\E_k$ and $\F_k$ with determinant $\pm 1$ such that $\E_k \L(\tau_k) \F_k = \diag( \P_k, \I_{N-n})$, for $0 \le k \le \ell$.  Interpolating these matrices polynomially (indeed, we already know the correct barycentric weights) to form polynomial matrices $\E(z)$ and $\F(z)$ we have $\E(z) \L(z) \F(z) = \diag(\P(z), \I_{N-m})$, and moreover because the determinants are $\pm1$ exactly at the nodes (and therefore nonzero in certain small neighbourhoods of the nodes) we see that $\L(z)$ and $\P(z)$ are equivalent on a set $\Sigma$ (not necessarily simply connected) that contains the nodes $\tau_k$.  Since $\L(z)$ and $\P(z)$ are also equivalent on a set that contains everything \emph{except} the nodes, we see by Proposition 2.1 of~\cite{Dopico2020} that $\L(z)$ is a linearization of $\P(z)$.

The $\X$ and $\Y$ for the standard triple are
\begin{align}
	\X &= 
    \begin{bmatrix}
    	0 & 1 & 1 & 1 & 1 & 1 
    \end{bmatrix} \>, \\
    \Y &=
    \begin{bmatrix}
    	1 & 0 & 0 & 0 & 0 & 0
    \end{bmatrix}^{\mathrm{T}} \>.
\end{align}
Notice in this case that for the linearization $N = (\ell+2)n$ while $\mathrm{deg}\ p \leq \ell$, and therefore, there are at least $2n$ eigenvalues at infinity. This can be inconvenient if $n$ is at all large.

\subsection{Hermite interpolational basis} \label{hermite-subsec}
The Lagrange linearization of the previous section has been extended to Hermite interpolational bases, where some of the nodes have ``flowed together'', collapsing to fewer distinct nodes\footnote{A formal definition can be found in~\cite{corless2013polynomial}, for instance. The essential idea is that given two distinct pieces of data $(\tau_k,p(\tau_k))$ and $(\tau_{k+1},p(\tau_{k+1}))$, we also know the forward difference $(p_{k+1}-p_k)/(\tau_{k+1}-\tau_k)$.  In the limit as one node approaches (flows towards) the other, we still know two pieces of information: $p(\tau_k)$ and $p'(\tau_k)$.  Hermite interpolation captures this idea.}. 

We suppose that at each remaining distinct node $\tau_{i}$, $0 \le i \le N-1$, say, there are now $s_{i} \geq 1$ consecutive pieces of information known, namely $\P(\tau_i)$, $\P'(\tau_i)/1!$, $\P''(\tau_i)/2!$, and so on up to the last one, the value of the $s_i-1$-th derivative at $z=\tau_i$, namely $\P^{(s_i-1)}(\tau_i)/(s_i-1)!$. The integer $s_{i}$ is called the \textit{confluency} of the node. The known pieces of information are the local Taylor coefficients of the polynomial fitting the data:
\begin{equation}
	\rho_{i, j} = \dfrac{f^{(j)}(\tau_{i})}{j!} \>, \quad 0 \leq j \leq s_{i} - 1 \>.
\end{equation}
This gives $1+\ell = \sum s_i$ pieces of information, determining a polynomial of grade~$\ell$.
The barycentric weights, this time doubly indexed as $\beta_{i,j}$, are again computed from the partial fraction decomposition of the reciprocal of the node polynomial
\begin{equation}
    \frac{1}{w(z)} = \frac{1}{\prod_{i=0}^{N-1} \left(z-\tau_i\right)^{s_i}} = \sum_{i=0}^{N-1} \sum_{j=0}^{s_i-1} \frac{\beta_{i,j}}{(z-\tau_i)^{j+1}}\>.
    \label{eq:Hermitenodepoly}
\end{equation}
For evaluation of the interpolating polynomial, one should use the first or second barycentric form; see~\cite{corless2013polynomial} for details.  For theoretical work with the Hermite interpolational bases, however, we can define
\begin{equation}
    H_{i,j}(z)
           = \sum_{k=0}^{s_i-1-j} \beta_{i,j+k} w(z)(z-\tau_i)^{-k-1} \label{eq:ExplicitHermiteForm2}\>.
\end{equation}
These polynomials, each of degree~$\ell$, form a basis (a Hermite interpolational basis, to distinguish from the Hermite orthogonal polynomials) for polynomials of grade~$\ell$; moreover, they generalize the Lagrange property in that only one Taylor coefficient at only one node is $1$ and all the rest are zero.

Note that the derivative $\P'(z)$ of a matrix polynomial is a straightforward extension to matrices of the ordinary derivative.  It is isomorphic to the matrix with entries that are the ordinary derivatives of the original matrix.

The linearization of the previous section changes to the following elegant form. The matrix $\C_{1}$ is unchanged,
\begin{equation}
	\C_{1} =
    \begin{bmatrix}
    	0 & & & & \\
        & 1 & & & \\
        & & \ddots & & \\
        & & & 1 & \\
        & & & & 1
    \end{bmatrix} \>,
\end{equation}
being $(\ell + 2)$ by $(\ell + 2)$ as before.
The matrix $\C_{0}$ changes, picking up transposed Jordan-like blocks for each distinct node. For instance, suppose we have two distinct nodes, $\tau_{0}$ and $\tau_{1}$. Suppose further that $\tau_{0}$ has confluency $s_{0} = 3$ while $\tau_{1}$ has confluency $s_{1} = 2$. This means that we know $f(\tau_{0})$, $\sfrac{f'(\tau_{0})}{1!}$, $\sfrac{f''(\tau_{0})}{2!}$, $f(\tau_{1})$ and $\sfrac{f'(\tau_{1})}{1!}$. Then,
\begin{equation}
	\C_{0} =
    \left[
    \begin{array}{cccccc}
    	0 & -\sfrac{f''(\tau_{0})}{2!} & -\sfrac{f'(\tau_{0})}{1!} & -f(\tau_{0}) & -\sfrac{f'(\tau_{1})}{1!} & -f(\tau_{1}) \\
        \cdashline{2-4}
        \beta_{02} & \multicolumn{1}{:c}{\tau_{0}} & & \multicolumn{1}{c:}{ }& & \\
        \beta_{01} & \multicolumn{1}{:c}{1} & \tau_{0}  & \multicolumn{1}{c:}{ } & & \\
        \beta_{00} & \multicolumn{1}{:c}{ } & 1 & \multicolumn{1}{c:}{\tau_{0}} & & \\
        \cdashline{2-6}
        \beta_{11} & & & & \multicolumn{1}{:c}{\tau_{1}} & \multicolumn{1}{c:}{} \\
        \beta_{10} & & & & \multicolumn{1}{:c}{1} & \multicolumn{1}{c:}{\tau_{1}} \\
        \cdashline{5-6}
    \end{array}
    \right] \>.
\end{equation}
Note the reverse ordering of the derivative values in this formulation.

The matrices $\E(z)$ and $\F(z)$ demonstrating that this is indeed a local linearization have, so far as we know, not been noted in the literature.  However, they are exactly the same as those for the Lagrange basis, mutatis mutandis, 
which are discussed in the appendix to~\cite{amiraslani2008linearization},with appropriately modified meanings for $\phi$ and $\D$.  The new $\phi$ contains the Hermite interpolational bases in Equation~\eqref{eq:ExplicitHermiteForm2},
and now $\D$ is not diagonal, but rather block diagonal with the transposed Jordan-like blocks above.  Both (rational) matrices are still unimodular.  Again we have $\E(z)(z\C_1-\C_0)\F(z) = \diag(\P(z), \I_n, \ldots, \I_n)$, everywhere except the nodes.

To generalize the argument of the previous section that the companion pencil is equivalent also at the nodes requires more work than the Lagrange case does, because the Jordan-like block structure interferes with the perturbation argument; on the other hand, one only needs the \emph{highest derivatives} at each node to be nonsingular for the construction to work.  We leave the details for another paper.  

We believe that the strict equivalence proof for the Lagrange basis in~\cite{corless2021equivalence} can be extended to the Hermite case as well, but this has not yet been carried out.

For the standard triple, take in the scalar case
\begin{equation}
	\Y =
    \begin{bmatrix}
    	1 & 0 & \cdots & 0
    \end{bmatrix}^{\mathrm{T}}
\end{equation}
but for $\X$ take the coefficients of the expansion of the polynomial $1$ in this particular Hermite interpolational basis: it is equal to $1$ at each node but has all derivatives zero at each node. That is, put
\begin{equation}
	\begin{cases}
    	\rho_{ij} = 1 \quad &\text{if } j = 0 \>, \\
        0 \quad &\text{otherwise} \>,
    \end{cases}
\end{equation}
and sort them in order:
\begin{equation}
	\X =
    \begin{bmatrix}
    	0 & \rho_{0, s_{0} - 1} & \rho_{0, s_{0} - 2} & \cdots & \rho_{0, 0} & \rho_{1, s_{1} - 1} & \cdots & \rho_{n, 0}
    \end{bmatrix} \>.
\end{equation}
For the earlier instance (two nodes, of confluency 3 and 2, respectively),
\begin{equation}
	\X = 
    \begin{bmatrix}
    	0 & \smash[b]{\underbrace{\begin{matrix}0 & 0 & 1 \end{matrix}}_\text{for $\tau_{0}$}} & \smash[b]{\underbrace{\begin{matrix}0 & 1\end{matrix}}_\text{for $\tau_{1}$}}
    \end{bmatrix} \>.
\end{equation}
Then,
\begin{equation}
	p^{-1}(z) = \X(z\C_{1} - \C_{0})^{-1}\Y \>.
\end{equation}

\begin{remark}
	We may re-order the nodes in any fashion we like, and each ordering generates its own linearization (both Hermite and Lagrange). We may also find a linearization where the confluent data is ordered $p(\tau_{i})$, $\sfrac{p'(\tau_{i})}{1!}$, $\sfrac{p''(\tau_{i})}{2!}$, etc., although we have not done so.
    
    If there is just one node of confluency $\ell$, we recover the standard Frobenius companion (plus two infinite roots):
    \begin{equation}
    	\begin{bmatrix}
        	0 & & & & \\
            & 1 & & & \\
            & & \ddots & & \\
            & & & 1 & \\
            & & & & 1
        \end{bmatrix}
        \quad , \quad
        \begin{bmatrix}
        	0 & -\scof{\ell-1} & -\scof{\ell-2} & \cdots & -\scof{1} & -\scof{0} \\
            1 & \tau_{0} & & & & \\
            0 & 1 & \tau_{0} & & & \\
            0 & & 1 & \ddots & & \\
            \vdots & & & \ddots & \tau_{0} & \\
            0 & & & & 1 & \tau_{0}
        \end{bmatrix}\>.
    \end{equation}
    Here, $\scof{k} = \sfrac{p^{(k)}(\tau_{0})}{k!}$ is the ordinary coefficient in the expansion $p(z) = \sum_{k=0}^{\ell} \scof{k}(z - \tau_{0})^{k}$. The numerical stability of these Hermite interpolational linearization has been studied briefly \cite{lawrence2012numerical} but much remains unknown. We confine ourselves in this paper to the study of the standard triple.
    
\end{remark}
To make a linearization for matrix polynomials out of these scalar linearizations, take the Kronecker tensor product with $\I_n$, and insert the appropriate matrix polynomial values and derivative values.

\begin{example}[Matrix polynomial case]
Let
\begin{equation*}
    \tau = \left[0, 1\right]
\end{equation*}
and
\begin{center}
    \begin{tabular}{c|c|c}
        $z$ & $\P(z)$ & $\P'(z)$ \\
        \hline
        $\tau_0 = 0$ & 
        $\begin{bmatrix}
            -1&0\\ 
            -1&1
        \end{bmatrix}$ & \\
        $\tau_1 = 1$ & 
        $\begin{bmatrix}
            0&\phantom{-}1\\ 
            1&-1
        \end{bmatrix}$ &
        $\begin{bmatrix}
            \phantom{-}1&-1\\ 
            -1&\phantom{-}0
        \end{bmatrix}$
    \end{tabular}
\end{center}
Then, the standard triple is
\begin{gather*}
	\C_{0} =
	\begin{bmatrix}
        0&0&-1&1&0&-1&1&0\\ 
        0&0&1&0&-1&1&1&-1\\ 
        1&0&1&0&0&0&0&0\\
        0&1&0&1&0&0&0&0\\
        -1&0&1&0&1&0&0&0\\
        0&-1&0&1&0&1&0&0\\
        1&0&0&0&0&0&0&0\\
        0&1&0&0&0&0&0&0
	\end{bmatrix} \quad
	\C_{1} =
	\begin{bmatrix}
		0&0&0&0&0&0&0&0\\ 
		0&0&0&0&0&0&0&0\\ 
		0&0&1&0&0&0&0&0\\ 
		0&0&0&1&0&0&0&0\\ 
		0&0&0&0&1&0&0&0\\ 
		0&0&0&0&0&1&0&0\\ 
		0&0&0&0&0&0&1&0\\ 
		0&0&0&0&0&0&0&1
	\end{bmatrix} \\
	\X = 
	\begin{bmatrix}
		0&0&0&0&1&0&1&0\\ 
		0&0&0&0&0&1&0&1
	\end{bmatrix} \quad
	\Y =
	\begin{bmatrix}
		1&0\\ 
		0&1\\ 
		0&0\\ 
		0&0\\ 
		0&0\\ 
		0&0\\ 
		0&0\\ 
		0&0
	\end{bmatrix} \>.
\end{gather*}
The Hermite interpolating polynomial is
\begin{equation*}
    \P(z) = 
    \begin{bmatrix}
        z-1&-2{z}^{2}+3z\\ 
        -3{z}^{2}+5z-1&2{z}^{2}-4z+1
    \end{bmatrix}
\end{equation*}
and the resolvent form is
\begin{equation*}
    \X\left(z\C_1 - \C_0\right)^{-1}\Y = 
    \begin{bmatrix}
        {\dfrac{-2{z}^{2}+4z-1}{6{z}^{4}-21{z}^{3}+23{z}^{2}-8z+1}}&{\dfrac{-2{z}^{2}+3z}{6{z}^{4}-21{z}^{3}+23{z}^{2}-8z+1}}\\ 
        {\dfrac{-3{z}^{2}+5z-1}{6{z}^{4}-21{z}^{3}+23{z}^{2}-8z+1}}&{\dfrac{-z+1}{6{z}^{4}-21{z}^{3}+23{z}^{2}-8z+1}}
    \end{bmatrix} \>.
\end{equation*}
Then,
\begin{equation*}
    \X\left(z\C_1 - \C_0\right)^{-1}\Y\P(z) = \I_{2} \>,
\end{equation*}
which indicates that the standard triples is correct.
\end{example}

\begin{remark}
The modified linearizations of~\cite{van2015linearization} also have standard triples that can be used for algebraic linearization, and arguably should be tabled here as well. They have the advantage of including fewer eigenvalues at infinity, or no spurious eigenvalues at infinity, which may lead to better algebraic linearizations. However, they are more involved, and we have less numerical experience with them. In particular, we do not understand their dependence on the ordering of the nodes, and so we leave their analysis to a future study.
\end{remark}

\section{Concluding remarks} \label{conclusion-sec}
The generalized standard triples (or standard quadruples, if you prefer) that we propose in this paper for convenience in algebraic linearization may have other uses.  As pointed out on p.~28 of~\cite{gohberg2009matrix} many of the properties stated in that work for monic polynomials are valid for non-monic polynomials with the appropriate changes made.  
Some caution with the results of this paper are thus mandated.

We have here \textsl{defined} these generalized standard triples simply by the resolvent representation for the matrix polynomial Equation~\eqref{eq:resolventrepresentation}, and only for linearizations, which is all we need for algebraic linearization.  

The main theorem of the paper, namely Theorem~\ref{thm:maintheorem}, gives a universal way to construct this generalized standard triple in any polynomial basis.  We also gave explicit instructions for this construction using any of several polynomial bases, for convenience, together with separate proofs using the Schur complement, which may give insight for further work in this area.

In Section~\ref{sec:AlgebraicLinearization}, we have sketched an explicit construction for matrices $\E(z)$ and $\F(z)$ showing that \textsl{algebraic} linearizations are, in fact, linearizations, with $\E(z)(z\D_H-\H)\F(z) = \diag(\P(z), \I, \ldots, \I)$.
We have also given new constructions for matrices $\E(z)$ and $\F(z)$ which likewise show that the companions for the Lagrange and Hermite interpolational bases are, in fact, linearizations (of matrix polynomials of higher grade). A proof for Lagrange interpolational bases was given already in~\cite{amiraslani2008linearization}, where indeed the linearization was proved to be strong, but the result for Hermite interpolational bases is new to this paper. 
We also used Hermite Form computations to give a new (to us) pair $\E(z)$ and $\F(z)$ for the ordinary monomial basis.

\section*{Acknowledgments}
We acknowledge the support of Western University, The National Science and Engineering Research Council of Canada, the Ontario Graduate Scholarship (OGS) program, the University of Alcal{\'{a}}, the Ontario Research Centre of Computer Algebra, and the Rotman Institute of Philosophy. Part of this work was developed while RMC was visiting the University of Alcal{\'a}, in the frame of the project Giner de los Rios. The authors would also like to thank Peter Lancaster for teaching RMC long ago the value of the $5 \times 5$ example. Similarly we thank John C.~Butcher for the proper usage of the word ``interpolational''. We thank Fran{\c{c}}oise Tisseur for her thorough comments on an earlier version of this paper, and likewise Froil\'an Dopico for a similarly fruitful discussion.  Finally, we thank an anonymous referee for improving (and correcting!) several proofs.

\bibliographystyle{plain}
\bibliography{references}

\end{document}